\documentclass{amsart}
\usepackage{amssymb}

\newcommand{\supp}{\mathsf{supt}}

\newcommand{\w}{\omega}
\newcommand{\Rel}{\mathsf{Rel}}
\newcommand{\FRel}{\mathsf{FRel}}
\newcommand{\Sym}{\mathsf{Sym}}
\newcommand{\FSym}{\mathsf{FSym}}
\newcommand{\End}{\mathsf{End}}

\newcommand{\F}{\mathcal F}

\newcommand{\Zeta}{\mathfrak Z_s}
\newcommand{\Cent}{\mathsf C}

\newtheorem{theorem}{Theorem}[section]
\newtheorem{lemma}[theorem]{Lemma}
\newtheorem{proposition}[theorem]{Proposition}
\newtheorem{corollary}[theorem]{Corollary}
\newtheorem{problem}[theorem]{Problem}
\theoremstyle{definition}
\newtheorem{definition}[theorem]{Definition}

\title[Perfectly supportable semigroups]{Perfectly supportable semigroups are $\sigma$-discrete in each Hausdorff shift-invariant topology}
\author{Taras Banakh and Igor Guran}
\address{Jan Kochanowski University in Kielce (Poland), and Ivan Franko National University of Lviv (Ukraine)}
\email{t.o.banakh@gmail.com, igor\_guran@yahoo.com}
\begin{document}
\begin{abstract} In this paper we introduce perfectly supportable semigroups and prove that they are  $\sigma$-discrete in each Hausdorff shift-invariant topology. The class of perfectly supportable semigroups includes each semigroup $S$ such that $\FSym(X)\subset S\subset\FRel(X)$ where $\FRel(X)$ is the semigroup of finitely supported relations on an infinite set $X$ and $\FSym(X)$ is the group of finitely supported permutations of $X$.
\end{abstract}

\maketitle

\section{Introduction}

The problem considered in this paper traces its history back to S.Ulam who  asked in \cite[p.178]{Scot} and \cite{Ulam} if for some infinite set $X$ the group $\Sym(X)$ of bijections of $X$ carries a non-discrete locally compact group topology. The Ulam's problem was solved in negative in 1967 by E.Gaughan \cite{Gau}. This result motivated the following problem posed in \cite{CGGR}:

\begin{problem}\label{prob1} Let $X$ be a set of cardinality $|X|=\mathfrak c$ and let $\FSym(X)$ be the group of bijections $f:X\to X$ having finite support $\supp(f)=\{x\in X:f(x)\ne x\}$. Does $\FSym(X)$ admit a non-discrete  locally compact Hausdorff group topology? 
\end{problem}

In \cite{CGGR} it was shown that for any infinite set $X$ the group $\FSym(X)$ does not admit a compact
Hausdorff group topology. A negative answer to Problem~\ref{prob1} was given in \cite{BGP}, where the following theorem was proved:

\begin{theorem}[Banakh-Guran-Protasov]\label{t1} For any set $X$ the group $\FSym(X)$ is $\sigma$-discrete in each Hausdorff shift-invariant topology on $\FSym(X)$. Consequently, each Baire Hausdorff shift-invariant topology on $\FSym(X)$ is discrete.
\end{theorem}

In this paper we generalize Theorem~\ref{t1} to the class of perfectly supportable semigroups.
Such semigroups are defined in Section~\ref{s:supp}. Our main result is Theorem~\ref{main}, proved in Section~\ref{s:main}. It implies Corollary~\ref{cor1} saying that each perfectly supportable semigroup is $\sigma$-discrete in each Hausdorff shift-invariant topology. In Section~\ref{s:rel} we present an example of perfectly supportable semigroup $\FRel(X)$, which contains many other perfectly supportable (semi)groups as sub(semi)groups.

\section{$\supp$-Semigroups and $\supp$-perfect semigroups}\label{s:supp}

In this section we define the classes of semigroups called $\supp$-semigroups and $\supp$-perfect semigroups.

\begin{definition}\label{def1} A {\em $\supp$-semigroup} is a pair $(S,\supp)$ consisting of a semigroup $S$ and a function $\supp:S\to 2^X$ with values in the power-set $2^X$ of some set $X$, such that for each elements $f,g\in S$ we get:
\begin{enumerate}
\item $\supp(fg)\subset\supp(f)\cup\supp(g)$,
\item $fg=gf$ if $\supp(f)\cap\supp(g)=\emptyset$.
\end{enumerate}
The function $\supp:S\to 2^X$ is called the {\em support map} of the $\supp$-semigroup $(S,\supp)$. The set $\supp(S)=\bigcup_{a\in S}\supp(a)\subset X$ is called {\em the support of $S$}.
\end{definition}

A typical example of a $\supp$-semigroup is the group $\Sym(X)$ of all bijections $f:X\to X$ of a set $X$, endowed with the support map $\supp:f\mapsto\{x\in X:f(x)\ne x\}$.
In Section~\ref{s:rel} we shall describe another $\supp$-semigroup $\Rel(X)$, which contains $\Sym(X)$ (and many other semigroups) as a $\supp$-subsemigroup.

Definition~\ref{def1} implies the following proposition-definition.

\begin{proposition} Let $(S,\supp)$ be a $\supp$-semigroup with support $X=\supp(S)$.
\begin{enumerate}
\item For each infinite cardinal $\lambda$ the set $S_{<\lambda}=\{f\in S:|\supp(f)|<\lambda\}$ is a subsemigroup of $S$;
\item For each subset $F\subset X$ the set $S(F)=\{f\in S:\supp(f)\subset F\}$ is a subsemigroup of $S$;
\item For each subsemigroup $T\subset S$ the pair $(T,\supp|T)$ is a $\supp$-semigroup.
\end{enumerate}
\end{proposition}
\smallskip

\begin{definition} A $\supp$-semigroup $(S,\supp)$ with support set $X=\supp(S)$ is called a {\em $\supp$-finitary semigroup} if
\begin{enumerate}
\item each element $f\in S$ has finite support $\supp(f)$ and
\item for each finite subset $F\subset X$ the subsemigroup $S(F)=\{f\in S:\supp(f)\subset F\}$ of $S$ is finite.
\end{enumerate}
\end{definition}

This definition implies that each $\supp$-finitary semigroup $S$ coincides with its subsemigroup $S_{<\w}$ consisting of finitely supported elements of $S$.

A typical example of a $\supp$-finitary $\supp$-semigroup is the group $\FSym(X)=\Sym(X)_{<\w}$ of finitely supported bijections of a set $X$.
\smallskip

To define $\supp$-perfect semigroups, we need to recall some information on centralizers.

For a semigroup $S$, an element $f\in S$, and a subset $T\subset S$ by
$$\Cent(f)=\{g\in G:fg=gf\}\mbox{ \ and \ }\Cent(T)=\bigcap_{f\in T}\Cent(f)$$
we denote the {\em centralizers} of $f$ and $T$, respectively.

The condition (2) of Definition~\ref{def1} implies that $S(F)\subset \Cent(S(X{\setminus}F))$ for any finite subset $F\subset X$. For $\supp$-perfect semigroups the converse implication is also true.

\begin{definition}\label{def2} A $\supp$-semigroup $(S,\supp)$ is called a {\em $\supp$-perfect semigroup} if it is $\supp$-finitary and for each finite subset $F\subset X$ we get $S(F)=\Cent(S(X{\setminus}F))$.
\end{definition}

\begin{definition} A semigroup $S$ is called {\em perfectly supportable} if for some function $\supp:S\to 2^X$ the pair $(S,\supp)$ is a $\supp$-perfect semigroup.
\end{definition}

Let us discuss the relation of perfectly supportable semigroups with some other classes of semigroups.

\begin{definition} We shall say that a semigroup $S$
\begin{itemize}
\item is {\em locally finite} if each finite subset $F\subset S$ lies in a finite subsemigroup $T\subset S$;
\item has {\em finite double centralizers} if for any finite subset $F\subset S$ its double centralizer $\Cent(\Cent(F))$ is finite.
\end{itemize}
\end{definition}

Since the double centralizer $\Cent(\Cent(F))$ is a subsemigroup that contains $F$, each semigroup with finite double centralizers is locally finite.

\begin{proposition}\label{fdc} Each perfectly supportable semigroup has finite double centralizers and hence is locally finite.
\end{proposition}

\begin{proof} Let $\supp:S\to 2^X$ be a support map that turns $S$ into a $\supp$-perfect semigroup. We claim that for each finite subset $T\subset S$ its double centralizer $\Cent(\Cent(T))$ lies in the finite subsemigroup $S(F)$ where $F=\bigcup_{f\in T}\supp(f)$. Assuming that $\Cent(\Cent(T))\not\subset S(F)$ and taking into account that $S(F)=\Cent(S(X{\setminus}F))$, we can find an element $f\in \Cent(\Cent(T))\setminus \Cent(S(X{\setminus}F))$. Then for some $g\in S(X{\setminus}F)$ we get $fg\ne gf$.
On the other hand, the condition (2) of Definition~\ref{def1} guarantees that $g\in S(X{\setminus F})\subset \Cent(S(F))\subset \Cent(T)$ and hence $f\notin \Cent(\Cent(T))$, which is a desired contradiction.
\end{proof}

\section{The topological structure of $\supp$-perfect and perfectly supportable semigroups}\label{s:main}

In this section we shall study the topological structure of $\supp$-perfect and perfectly supportable semigroups endowed with shift-invariant topologies.

A topology $\tau$ on a semigroup $S$ is called {\em shift-invariant} if for every $a\in A$ the left and right shifts $$l_a:S\to \;\; l_a:x\mapsto ax\mbox{ \ \ and \ \ }r_a:S\to S,\;\;r_a:x\mapsto xa,$$ are continuous.
This is equivalent to saying that the semigroup operation $S\times S\to S$ is separately continuous.

Now on each semigroup $S$ we define a shift-invariant $T_1$-topology (called the semi-Zariski topology) which is weaker than each Hausdorff shift-invariant topology on $S$.

The {\em semi-Zariski topology} $\Zeta$ on a semigroup $S$ is the topology generated by the sub-base consisting of the sets
$$\{x\in S:axb\ne c\}\mbox{ \ and \ }\{x\in S:axb\ne cxd\}$$where $a,b,c,d\in S^1$ and $S^1=S\cup\{1\}$ stands for the semigroup $S$ with attached external unit $1\notin S$ (i.e., an element $1$ such that $1x=x1=x$ for all $x\in S^1$). The semi-Zariski topology $\Zeta$ is a particular case of algebraically determined topologies on G-acts, considered in \cite{BPS}.

The definition of the semi-Zariski topology implies the following simple (but important) fact.

\begin{proposition} The semi-Zariski topology $\Zeta$ on a semigroup $S$ is weaker than each shift-invariant Hausdorff topology on $S$.
\end{proposition}

Now, we shall study the semi-Zariski topology on $\supp$-perfect semigroups.
The following theorem is the main result of this paper.

\begin{theorem}\label{main} Let $(S,\supp)$ be a $\supp$-perfect semigroup endowed with the semi-Zariski topology $\Zeta$. For every $n\ge 0$
\begin{enumerate}
\item the set $S_{\le n}=\{f\in S:|\supp(f)|\le n\}$ is closed in $S$;
\item for every $x\in X$ the set $\{f\in S_{\le n}:x\in \supp(f)\}$ is open in $S_{\le n}$;
\item the set $S_{=n}=\{f\in S:|\supp(f)|=n\}$ is discrete in $S$;
\end{enumerate}
\end{theorem}

\begin{proof} First, we prove two lemmas.

\begin{lemma}\label{lem1} For each $f\in S$, each point $x\in\supp(f)$, and each finite subset $F\subset X\setminus\{x\}$ either $S(\{x\})\setminus \Cent(f)\ne\emptyset$ or else there is an infinite subset $\F\subset S(X{\setminus}F)\setminus \Cent(f)$ such that $\supp(g)\cap\supp(h)\subset\{x\}\ne\supp(g)$ for any distinct elements $g,h\in\F$.
\end{lemma}

\begin{proof} Assume that $S(\{x\})\setminus \Cent(f)=\emptyset$. By induction we shall construct a sequence of elements $(g_i)_{i\in\w}$ of the semigroup $S$ and an increasing sequence $(F_i)_{i\in\w}$ of finite subsets of $X$ such that $F_{0}=F$ and for every $i\in\w$ the following conditions hold:
\begin{enumerate}
\item $x\notin F_{i}$,
\item $fg_i\ne g_if$,
\item $\supp(g_i)\cap F_i=\emptyset$ and $\supp(g_i)\not\subset\{x\}$,
\item $F_{i+1}=F_{i}\cup\supp(g_i)\setminus\{x\}$.
\end{enumerate}

Assume that for some $i\ge 0$ the set $F_{i}$ has been constructed. Since $x\in\supp(f)$ and $x\notin F_{i}$, we get $f\notin S(F_{i})=\Cent(S(X{\setminus}F_i))$ and hence there is an element $g_i\in S(X{\setminus}F_{i})$ such that $fg_i\ne g_if$. Put $F_{i+1}=F_{i}\cup\supp(g_i)\setminus\{x\}$ and observe that  $S(\{x\})\setminus \Cent(f)=\emptyset$ implies $\supp(g_i)\not\subset\{x\}$. This completes the inductive construction.
\smallskip

It follows that for any number $i<j$ we get $$\supp(g_j){\setminus}\{x\}\subset X{\setminus} F_{j}\subset X{\setminus} F_{i+1}\subset X\setminus \big(\supp(g_i){\setminus}\{x\}\big),$$ which implies  that the non-empty sets $\supp(g_i){\setminus}\{x\}$ and $\supp(g_j){\setminus}\{x\}$ are disjoint. Then $g_i\ne g_j$ and $\supp(g_i)\cap\supp(g_j)\subset\{x\}$.

So, $\F=\{g_i\}_{i\in\w}$ is a required infinite set in $S{\setminus} \Cent(f)$ such that  $\supp(g)\cap\supp(h)\subset\{x\}\ne\supp(g)$ for any distinct elements $g,h\in\F$.
\end{proof}

\begin{lemma}\label{lem2} For every $n\in\w$, each point $f\in S$ has a neighborhood $O_f\in\Zeta$ in the semi-Zariski topology such that for each $g\in O_f\cap S_{\le n}$ we get $\supp(f)\subset\supp(g)$.
\end{lemma}

\begin{proof} Let $m=|\supp(f)|$ and $\supp(f)=\{x_1,\dots,x_m\}$ be an enumeration of the finite set $\supp(f)$. For every $i\le m$ by induction we shall construct an increasing sequence of finite sets $F_0\subset F_1\subset\dots\subset F_m$ in $X$ and a sequence $\F_1,\dots,\F_m$ of non-empty finite subsets of $S\setminus \Cent(f)$ such that for every positive $i\le m$ the following conditions are satisfied:
\begin{enumerate}
\item either $\F_i\subset S(\{x_i\})$ or $|\F_i|=n+1$;
\item $\supp(g)\cap F_{i-1}\subset\{x_i\}$ for each $g\in\F_i$;
\item $\supp(g)\cap\supp(h)\subset\{x_i\}$ for any distinct elements $g,h\in\F_i$;
\item $F_i=F_{i-1}\cup\bigcup_{g\in\F_i}\supp(g)$.
\end{enumerate}

We start the inductive construction letting $F_0=\supp(f)$. Assume that for some $i<m$ the finite set $F_{i-1}$ has been constructed. Let $F=F_{i-1}\setminus\{x_i\}$ and apply Lemma~\ref{lem1} to find a non-empty family $\F_i\subset S(X{\setminus}F)\setminus \Cent(f)$ which satisfies the conditions (1) and (3) of the inductive construction. It follows from $\F_i\subset S(X{\setminus}F)$ that the condition (2) is satisfies. Finally, define the finite set $F_i$ by the condition (4). This completes the inductive construction.
\smallskip

The family $\F=\bigcup_{i=0}^m\F_i\subset S\setminus \Cent(f)$ determines an open neighborhood
$$O_f=\{h\in S:\forall g\in\F\;\;hg\ne gh\}$$of $f$ in the semi-Zariski topology $\Zeta$.
We claim that $\{x_1,\dots,x_m\}=\supp(f)\subset\supp(h)$ for each element $h\in O_f\cap S_{\le n}$. Assume conversely that $x_i\notin\supp(h)$ for some $i\le m$ and consider two cases.

(i) If $\F_i\subset S(\{x_i\})$, then for each element $g\in\F_i$ we get $\supp(g)\cap\supp(h)\subset \{x_i\}\cap\supp(h)=\emptyset$ and hence $gh=hg$ by the condition (2) of Definition 1.

(ii) If $\F_i\not\subset S(\{x_i\})$, then $|\F_i|=n+1$ and by the conditions (3) of the inductive construction, the family $\big\{\supp(g)\setminus\{x_i\}\big\}_{g\in\F_i}$ is disjoint and consists of non-empty sets. Since $|\F_i|=n+1>|\supp(h)|$, there is an element $g\in\F_i$ such that $\supp(g)\setminus\{x_i\}$ is disjoint with $\supp(h)$. Since $x_i\notin\supp(h)$, the supports $\supp(g)$ and $\supp(h)$ are disjoint and hence $gh=hg$ by the condition (2) of Definition~\ref{def1}.

In both cases, we get an element $g\in\F$ with $gh=hg$, which contradicts the choice of $h\in O_f$.
\end{proof}
\smallskip

Now we can finish the proof Theorem~\ref{main}. Let $n\in\w$.
\smallskip

1. To show that the set $S_{\le n}=\{f\in S:|\supp(f)|\le n\}$ is closed in $S$, fix any element $f\in S\setminus S_{\le n}$. By Lemma~\ref{lem2}, the element $f$ has a neighborood $O_f\subset S$ in the semi-Zariski topology $\Zeta$ such that each $g\in O_f\cap S_{\le n}$ has support
$\supp(g)\supset \supp(f)$, which implies that $n\ge |\supp(g)|\ge|\supp(f)|>n$ and $O_f\cap S_{\le n}=\emptyset$. So, $S_{\le n}$ is closed in $S$.
\smallskip

2. The second item of Theorem~\ref{main} follows directly from Lemma~\ref{lem2}.
\smallskip

3. Finally we show that the set $S_{=n}=\{f\in S:|\supp(f)|=n\}$ is discrete in $(S,\Zeta)$. Fix any element $f\in S_{=n}$ and using Lemma~\ref{lem2}, find a neighborhood $O_f\in\Zeta$ of $f$ such that $\supp(f)\subset\supp(h)$ for each $h\in O_f\cap S_{=n}$. Since $|\supp(f)|=n=|\supp(h)|$, we conclude that $\supp(f)=\supp(g)$ and hence $O_f\cap S_{=n}$ lies in the semigroup $S(\supp(f))$, which is finite by the condition (2) of Definition~\ref{def2}. Since the semi-Zariski topology $\Zeta$ satisfies the separation axiom $T_1$, the open finite subspace $O_f\cap S_{=n}$ of $S_=n$ is discrete and hence $f$ is an isolated point of $S_{=n}$, which means that the space $S_{=n}$ is discrete.
\end{proof}

Let us recall that topological space $X$ is {\em $\sigma$-discrete} if it can be written as a countable union of discrete subspaces. A topology $\tau$ on a set $X$ is called {\em Baire} if
for any sequence $(U_n)_{n\in\w}$ of open dense subsets of the topological space $(X,\tau)$ the intersection $\bigcap_{n\in\w}U_n$ is dense in $X$. It is well-known that each $\sigma$-discrete Baire $T_1$-space has an isolated point.

 Theorem~\ref{main} implies the main corollary of this paper:

\begin{corollary}\label{cor1} Each perfectly supportable semigroup $S$ is $\sigma$-discrete in its semi-Zariski topology and hence is $\sigma$-discrete in each shift-invariant Hausdorff topology on $S$.
\end{corollary}

Another corollary concerns perfectly supportable groups. A group $G$ is called {\em perfectly supportable} if it is perfectly supportable as a semigroup.

\begin{corollary}\label{cor2} Each perfectly supportable group $G$ is discrete in each Baire shift-invariant Hausdorff topology on $G$.
\end{corollary}

\begin{proof} Let $\tau$ be a Baire Hausdorff shift-invariant topology on a perfectly supportable group $G$. Since $G$ is a group, the topological space $(G,\tau)$ is topologically invariant.

By Corollary~\ref{cor1}, the topological space $(G,\tau)$ is $\sigma$-discrete and being Baire, has an isolated point. The topological homogeneity of $(G,\tau)$ guarantees that each point of $G$ is isolated and hence the space $(G,\tau)$ is discrete.
\end{proof}

\section{An example of a $\supp$-perfect semigroup}\label{s:rel}

In this section we consider an important example of a $\supp$-semigroup (which contains many other $\supp$-semigroups as $\supp$-subsemigroups).

Given a set $X$, consider the semigroup $\Rel(X)=2^{X\times X}$ of all relations $f\subset X\times X$, endowed with the operation
$$f\circ g=\{(x,z)\in X\times X:\exists y\in X\;\;(x,y)\in f,\;(y,z)\in g\}$$of composition of relations.  For a relation $f\subset X\times X$ by $f^{-1}$ we denote the inverse relation $$f^{-1}=\{(y,x):(x,y)\in f\}.$$

Each relation $f\subset X\times X$ can be considered as a multi-valued function assigning to each point $x\in X$ the subset $f(x)=\{y\in X:(x,y)\in f\}$ and to each subset $A\subset X$ the subset $f(A)=\bigcup_{a\in A}f(a)$ of $X$. Observe that two relations $f,g\subset X\times X$ are equal if and only if $f(x)=g(x)$ for all $x\in X$.

Each function $f:X\to X$ can be identified with the relation $\{(x,f(x)):x\in X\}$. So, the semigroup $\End(X)$ of all self-maps of $X$ is a subsemigroup of the semigroup $\Rel(X)$. By the same reason, the symmetric group $\Sym(X)$ is a subgroup of the semigroup $\Rel(X)$.

For a relation $f\subset X\times X$ its support is defined as the subset
$$\supp(f)=\big\{x\in X:f(x)\ne\{x\}\mbox{ or }f^{-1}(x)\ne\{x\}\big\}\subset X.$$

\begin{proposition}\label{prop1} The pair $(\Rel(X),\supp)$ is a $\supp$-semigroup.
\end{proposition}

\begin{proof} Fix any relations $f,g\subset X\times X$.

To show that $\supp(f\circ g)\subset\supp(f)\cup\supp(g)$, take any point $x\in X$ that does not belong to the union $\supp(f)\cup\supp(g)$. Then
$$f\circ g(x)=f(\{x\})=\{x\}$$and
$$(f\circ g)^{-1}(x)=g^{-1}\circ f^{-1}(x)=g^{-1}(\{x\})=\{x\},$$which means that $x\notin\supp(f\circ g)$.

Next, assuming that $\supp(f)\cap \supp(g)=\emptyset$, we shall show that $f\circ g(x)=g\circ f(x)$ for each point $x\in X$. The inclusion $f\circ g(x)\subset g\circ f(x)$ will follow as soon as we prove that each point $y\in f\circ g(x)$ belongs to the set $g\circ f(x)$. Find a point $z\in g(x)$ such that $y\in f(z)$.

If $y\ne z$, then $f(z)\ne\{z\}$ and $f^{-1}(y)\ne\{y\}$, which implies $y,z\in\supp(f)\subset X\setminus \supp(g)$. Then $x\in g^{-1}(z)=\{z\}$ and hence $x=z$ and $$y\in \{y\}=g(y)\subset g\circ f(z)=g\circ f(x).$$

Now assume that $y=z$. If $z=x$, then $y=z\in g(x)=g(y)\subset g\circ f(x)$. If $z\ne x$, then $z\in g(x)\ne\{x\}$ and $x\in g^{-1}(z)\ne\{z\}$ and hence $x,z\in\supp(g)\subset X\setminus \supp(f)$.
Then $y=z\in g(x)=g(\{x\})=g\circ f(x)$.

This completes the proof of the inclusion $y\in g\circ f(x)$, which implies that $f\circ g(x)\subset g\circ f(x)$. By analogy we can prove that $g\circ f(x)\subset f\circ g(x)$.
\end{proof}

In the $\supp$-semigroup $\Rel(X)$ consider the $\supp$-subsemigroup
$$\FRel(X)=\{f\in\Rel(X):\supp(f)\mbox{ is finite}\}.$$
Observe that for each (finite) subset $F\subset X$ the subsemigroup $\{f\in\Rel(X):\supp(f)\subset F\}$ can be identified with the (finite) semigroup $\Rel(F)$, which has cardinality $2^{|F\times F|}$. This observaion implies:

\begin{proposition} The semigroup $\FRel(X)$ endowed with the support map $\supp:\FRel(X)\to 2^X$ is $\supp$-finitary.
\end{proposition}

Finally we check that the $\supp$-semigroup $\FRel(X)$ and some its subsemigroups are $\supp$-perfect.

\begin{proposition}\label{prop2} For an infinite  set $X$, a subsemigroup $S\subset \FRel(X)$ is $\supp$-perfect if for each finite subset $E\subset X$ and each point $x\in X\setminus E$ there is a relation $g\in S(X{\setminus}E)$ such that $x\notin g(x)\ne\emptyset$.
\end{proposition}

\begin{proof} Given a finite subset $F\subset X$ we should check that $S(F)=\Cent(S(X{\setminus}F))$. The inclusion $S(F)\subset \Cent(S(X{\setminus}F))$ holds because $S$ is a $\supp$-semigroup. To prove that $\Cent(S(X{\setminus}F))\subset S(F)$, take any element $f\in \Cent(S(X{\setminus}F))$ and assume that $f\notin S(F)$. Then $\supp(f)\not\subset F$ and we can choose a point $x\in \supp(f)\setminus F$. By our assumption, for the finite set $E=F\cup(\supp(f)\setminus\{x\})$ there is a relation $g\in S(X{\setminus}E)\subset S(X{\setminus}F)$ such that $x\notin g(x)\ne\emptyset$. Then $\{x\}\cup g(x)\subset\supp(g)$.

On the other hand, $x\in\supp(f)$ implies $\{x\}\cup f(x)\subset\supp(f)$.
Then $$f(x)\cap g(x)\subset\supp(f)\cap\supp(g)\subset\{x\}$$and hence
$f(g(x))=g(x)$ and $g(f(x)\setminus\{x\})=f(x)\setminus\{x\}$.

We claim that $f\circ g\ne g\circ f$. Assuming that $f\circ g=g\circ f$ and taking into account that  $x\notin g(x)\ne\emptyset$, we conclude that
$$g(x)=f(g(x))=g(f(x))\supset g(f(x)\setminus\{x\})=f(x)\setminus\{x\}$$
and hence $f(x)\setminus\{x\}\subset (f(x)\setminus\{x\})\cap g(x)\subset \supp(f)\cap\supp(g)\setminus\{x\}=\emptyset$, which implies that $f(x)\subset\{x\}$.

If $f(x)=\emptyset$, then the set $g(f(x))$ is empty while $f(g(x))=g(x)\ne\emptyset$. So, $f\circ g\ne g\circ f$.

So, $f(x)\ne\emptyset$ and hence $f(x)=\{x\}$. Then $x\in\supp(f)$ implies that $f^{-1}(x)\ne\{x\}$. Since $x\in f^{-1}(x)$, the set $f^{-1}(x)\ne\{x\}$ is not empty and hence it contains a point $y\in f^{-1}(x)\setminus\{x\}$. It follows that $\{y\}\cup f(y)\subset\supp(f)\subset \{x\}\cup(X\setminus \supp(g))$ and hence $g(y)=\{y\}$ and $x\in f(\{y\})=f(g(y))=g(f(y))=g(f(y)\setminus\{x\})\cup g(x)=(f(y)\setminus\{x\})\cup g(x)$, which contradicts $x\notin g(x)$.

This contradiction shows that $g\circ f\ne f\circ g$ and hence $f\notin \Cent(S(X{\setminus}F))$ as $g\in S(X{\setminus}E)\subset S(X{\setminus}F)$.
\end{proof}

Proposition~\ref{prop2} and Corollary~\ref{cor1} imply:

\begin{corollary}\label{cor3} Let $X$ be an infinite set and $S\subset\FRel(X)$ be a subsemigroup such that  for each finite subset $E\subset X$ and each point $x\in X\setminus E$ there is a relation $g\in S$ such that $x\notin g(x)\ne\emptyset$ and $\supp(g)\cap E\ne\emptyset$. Then the semigroup $S$ is $\sigma$-discrete is its semi-Zariski topology and consequently is $\sigma$-discrete in each Hausdorff shift-invariant topology on $S$.
\end{corollary}

This corollary implies the following theorem announced in Abstract.

\begin{theorem} Let $X$ be an infinite set $X$ and $S$ be a semigroup such that $\FSym(X)\subset S\subset\FRel(X)$. Then
\begin{enumerate}
\item $S$ is perfectly supportable;
\item $S$ is $\sigma$-discrete is its semi-Zariski topology $\Zeta$;
\item $S$ is $\sigma$-discrete in each Hausdorff shift-invariant topology on $S$.
\end{enumerate}
\end{theorem}

\section{Some Open Problems}

The second statement of Theorem~\ref{main} suggests the following question.

\begin{problem} Let $(S,\supp)$ be a $\supp$-perfect semigroup and $x\in X$.
Is the subset $\{f\in S:x\in\supp(f)\}$ open in the semi-Zariski topology $\Zeta$? In each Hausdorff semigroup topology on $S$?
\end{problem}

By Proposition~\ref{fdc}, each perfectly supportable semigroup has finite double centralizers. We do not know if the converse is also true.

\begin{problem} Is each (semi)group with finite double centralizers perfectly supportable?
\end{problem}

The affirmative answer to this problem would imply affirmative answers to the following two problems:

\begin{problem} Is each (semi)group with finite double centralizers $\sigma$-discrete in its semi-Zariski topology?
\end{problem}

\begin{problem} Let $S$ be a semigroup with finite double centralizers and $n\in\mathbb N$. Is the set $S_{\le n}=\{f\in S:|\Cent(\Cent(f))|\le n\}$ closed in the semi-Zariski topology of $S$? Is it $\sigma$-discrete?
\end{problem}


\begin{thebibliography}{}

\bibitem{BGP} T.~Banakh, I.~Guran, I.~Protasov, {\em Algebraically determined topologies on permutation groups}, preprint.

\bibitem{BPS} T.~Banakh, I.~Protasov, O.~Sipacheva, {\em Topologization of sets endowed with an action of a monoid}, preprint.

\bibitem{Gau} E.~Gaughan, {\em Group structures of infinite symmetric groups}, Proc. Nat. Acad. Sci. U.S.A. {\bf 58} (1967), 907--910.

\bibitem{CGGR} I.~Guran, O.~Gutik, O.~Ravsky, I.~Chuchman, {\em On symmetric topologiacl semigroups and groups}, Visnyk Lviv Univ. Ser. Mech. Math. {\bf 74} (2011), 61--73 (in Ukrainian).

\bibitem{Scot} D.Mauldin (ed.), {\em The Scottish Book. Mathematics from the Scottish Caf\'e}, Birkhauser, Boston, Mass., 1981.

\bibitem{Ulam} S.~Ulam, {\em A Collection of Mathematical Problems}, Intersci. Publ., NY, 1960.

\end{thebibliography}
\end{document}